\documentclass[a4paper,11pt,reqno]{amsart}
\usepackage{amsmath,amssymb,url,stmaryrd}
\usepackage[T1]{fontenc}
\usepackage{lmodern}
\usepackage[top=30truemm,bottom=30truemm,left=20truemm,right=20truemm]{geometry}
\usepackage{tikz}
\usetikzlibrary{positioning}
\usetikzlibrary{arrows.meta}
\usetikzlibrary{calc}
\usepackage{color}
\usepackage{enumitem}
\usepackage{hyperref}

\newenvironment{enumerate2}{\begin{enumerate}[label=\textup{(\alph*)}]}
{\end{enumerate}}

\date{\today}

\newtheorem{Thm}{Theorem}[section]
\newtheorem{Cor}[Thm]{Corollary}
\newtheorem{Lem}[Thm]{Lemma}
\newtheorem{Prop}[Thm]{Proposition}
\newtheorem{Def-Lem}[Thm]{Definition-Lemma}
\newtheorem{Def-Prop}[Thm]{Definition-Proposition}

\theoremstyle{definition}
\newtheorem{Def}[Thm]{Definition}

\newcommand{\calA}{\mathcal{A}}

\newcommand{\calF}{\mathcal{F}}

\newcommand{\calO}{\mathcal{O}}
\newcommand{\calS}{\mathcal{S}}

\newcommand{\calU}{\mathcal{U}}
\newcommand{\calV}{\mathcal{V}}

\newcommand{\calX}{\mathcal{X}}
\newcommand{\calY}{\mathcal{Y}}
\newcommand{\calZ}{\mathcal{Z}}

\newcommand{\ovcalF}{\overline{\calF}}

\newcommand{\sfD}{\mathsf{D}}
\newcommand{\sfK}{\mathsf{K}}

\newcommand{\rmb}{\mathrm{b}}

\newcommand{\Z}{\mathbb{Z}}

\DeclareMathOperator{\Hom}{\mathsf{Hom}}
\DeclareMathOperator{\End}{\mathsf{End}}

\DeclareMathOperator{\Ext}{\mathsf{Ext}}

\DeclareMathOperator{\Coker}{\mathsf{Coker}}

\DeclareMathOperator{\module}{\mathsf{mod}} \renewcommand{\mod}{\module}
\DeclareMathOperator{\rep}{\mathsf{rep}}

\DeclareMathOperator{\proj}{\mathsf{proj}}

\DeclareMathOperator{\add}{\mathsf{add}}

\DeclareMathOperator{\brick}{\mathsf{brick}}
\DeclareMathOperator{\sbrick}{\mathsf{sbrick}}

\DeclareMathOperator{\sfdim}{\mathsf{dim}}
\renewcommand{\dim}{\sfdim}
\DeclareMathOperator{\dimv}{\underline{\mathsf{dim}}}

\DeclareMathOperator{\GL}{\mathsf{GL}}

\DeclareMathOperator{\Irr}{\mathsf{Irr}}

\renewcommand{\epsilon}{\varepsilon}
\renewcommand{\Gamma}{\varGamma}
\renewcommand{\Lambda}{\varLambda}
\renewcommand{\Phi}{\varPhi}
\renewcommand{\phi}{\varphi}

\numberwithin{equation}{section}

\begin{document}
\title{Maximal finite semibricks consist only of open bricks}

\author{Sota Asai} 
\address{Sota Asai: Graduate School of Mathematical Sciences,
the University of Tokyo,  
3-8-1 Komaba, Meguro-ku, Tokyo-to, 153-8914, Japan}
\email{sotaasai@g.ecc.u-tokyo.ac.jp}

\begin{abstract}
A semibrick is a set of modules satisfying Schur's Lemma,
and it is said to be maximal 
if it is not properly contained in another semibrick.
For any finite dimensional algebra $\Lambda$ 
over an algebracally closed field $K$,
we prove that any maximal finite semibrick $\calS$ 
consists only of open bricks $B$, that is, 
bricks whose orbit closures $\overline{\calO_B}$ 
are irreducible components in the representation schemes.
\end{abstract}

\maketitle 

\section{Our results}\label{Sec_result}

In this short paper, we discuss maximal finite semibricks 
in the module category $\mod \Lambda$
of a finite dimensional algebra $\Lambda$ 
over an algebraically closed field $K$.
Bricks, semibricks and related notions are defined as follows:

\begin{itemize}
\item
A $\Lambda$-module $B \in \mod \Lambda$ is called a \emph{brick}
if the endomorphism ring $\End_\Lambda(B)$ is isomorphic to $K$.
We write $\brick \Lambda$ for the set of isoclasses of bricks 
in $\mod \Lambda$.
\item
\cite{A-semi}
A subset $\calS \subset \brick \Lambda$ is called a \emph{semibrick}
if $\Hom_\Lambda(B,B')=0$ for any $B \ne B' \in \calS$.
We write $\sbrick \Lambda$ for the set of semibricks in $\mod \Lambda$.
\item
A semibrick $\calS \in \sbrick \Lambda$ is said to be \emph{finite} 
if $\calS$ is a finite set.
\item
A semibrick $\calS \in \sbrick \Lambda$ is said to be \emph{maximal} 
if there exists no semibrick $\calS'$ 
such that $\calS \subsetneq \calS'$.
\end{itemize}

It is well-known that there exist a finite quiver $Q$ 
and an admissible ideal $I \subset KQ$ such that
$\Lambda$ is Morita equivalent to $KQ/I$,
so we may assume that $\Lambda=KQ/I$.

To each $d=(d_i)_{i \in Q_0} \in (\Z_{\ge 0})^{Q_0}$,
we associate a scheme $\rep(\Lambda,d)$ 
called the \emph{representation scheme} (see Definition \ref{Def_rep_var}),
which can be seen as the set of all $\Lambda$-modules $M$ 
whose dimension vectors $\dimv M$ are $d$.
We write $\Irr(\Lambda,d)$ for the set of irreducible components 
of $\rep(\Lambda,d)$,
and set $\Irr(\Lambda)$ as the disjoint union of all $\Irr(\Lambda,d)$
for $d \in (\Z_{\ge 0})^{Q_0}$.

The group $\GL(d)=\prod_{i \in Q_0}\GL(d_i)$ acts naturally 
on $\rep(\Lambda,d)$,
and for each $M \in \rep(\Lambda,d)$, its orbit $\calO_M$ coincides with
the subset 
$\{X \in \rep(\Lambda,d) \mid \text{$X \simeq M$ as $\Lambda$-modules}\}$.
Then every irreducible component of $\rep(\Lambda,d)$ is 
a union of orbits, since $\GL(d)$ is connected.

For a brick $B$,
we say that a brick $B$ is an \emph{open brick}
if $\calO_B$ is open dense in some $\calZ \in \Irr(\Lambda)$;
otherwise, $B$ is called a \emph{non-open brick}.
The following is the main result of this paper.

\begin{Thm}\label{Thm_open}
Let $\Lambda$ be a finite dimensional algebra 
over an algebracally closed field $K$.
Assume that $\calS \in \sbrick \Lambda$ is a maximal finite semibrick.
Then every brick in $\calS$ is an open brick.
\end{Thm}

Theorem \ref{Thm_open} is a direct consequence 
of the following general result,
which is \emph{an extension theorem of semibricks}.

\begin{Thm}\label{Thm_main}
Let $\Lambda$ be a finite dimensional algebra 
over an algebracally closed field $K$,
and $\calS \in \sbrick \Lambda$ be a finite semibrick.
If $B \in \calS$ and $\calZ \in \Irr(\Lambda)$ satisfy 
$\overline{\calO_B} \subsetneq \calZ$,
then there exists some brick $B'$ satisfying
\begin{align*}
\calS \sqcup \{B'\} \in \sbrick \Lambda, 
\quad B' \in \calZ.
\end{align*}
\end{Thm}

Applying Theorem \ref{Thm_main} repeatedly, we get the following.

\begin{Cor}\label{Cor_infin}
Let $\Lambda$ be a finite dimensional algebra 
over an algebracally closed field $K$,
and $\calS \in \sbrick \Lambda$ be a finite semibrick.
Let $\calZ_1,\ldots,\calZ_m \in \Irr(\Lambda)$ with $m \ge 1$
satisfy that,
for each $i \in \{1,\ldots,m\}$, there exists a non-open brick $B \in \calS$
such that $B \in \calZ_i$.
Then there exist infinite semibricks $\calS_1,\ldots,\calS_m$ satisfying
\begin{align*}
\calS \sqcup \calS_1 \sqcup \cdots \sqcup \calS_m 
\in \sbrick \Lambda, \quad 
\calS_i \subset \calZ_i \quad (i \in \{1,\ldots,m\}).
\end{align*}
\end{Cor}

The next section is devoted to showing Theorem \ref{Thm_main}.

In Section \ref{Sec_old}, we prove
a weaker version of Proposition \ref{Prop_perp} for path algebras
(Proposition \ref{Prop_perp_weak})
by using semi-invariants as a record, 
which was our first proof of Theorem \ref{Thm_open}.

\section{Proofs of main results}

\subsection{Preliminaries}

Throughout this paper, 
$K$ is an algebraically closed field,
and $\Lambda=KQ/I$ is a finite dimensional algebra over $K$
with $Q=(Q_0,Q_1)$ a finite quiver and $I$ an admissible ideal of $KQ$.
All finite dimensional algebras over $K$ are described in this way,
up to Morita equivalences.

For each arrow $\alpha \colon i \to j$ in $Q$,
we set $s(\alpha):=i$ and $t(\alpha):=j$,
so $s(\alpha)$ and $t(\alpha)$ are the source and the target of $\alpha$.
The composite of two arrows $\alpha \colon i \to j$ and $\beta \colon j \to k$
are denoted by $\alpha \beta$.

For any module $M \in \mod \Lambda$, 
the element $\dimv M:=(\dim_K Me_i)_{i \in Q_0} \in (\Z_{\ge 0})^{Q_0}$ 
is called the \emph{dimension vector} of $M$,
where $e_i \in \Lambda$ is the idempotent corresponding to each $i \in Q_0$.

Now we recall the definition and some properties of 
representation schemes (also known as representation varieties) 
necessary in this paper.
See \cite{CB3} for details.

\begin{Def}\label{Def_rep_var}
Let $d=(d_i)_{i \in Q_0} \in (\Z_{\ge 0})^{Q_0}$.
Then we define the \emph{representation schemes} 
$\rep(Q,d)$ and $\rep(\Lambda,d)$ of $Q$ and $\Lambda$ for $d$ by  
\begin{align*}
\rep(Q,d)&:=\prod_{\alpha \in Q_1} 
\Hom_K(K^{d_{s(\alpha)}},K^{d_{t(\alpha)}}), \\
\rep(\Lambda,d)&:=
\left\{ (\phi_\alpha)_{\alpha \in Q_1} \in \rep(Q,d) 
\mid \text{$(\phi_\alpha)_{\alpha \in Q_1}$ 
satisfies the relations corresponding to $I$}
\right\}.
\end{align*}
Both are algebraic schemes equipped with the Zariski topology.
We set $\Irr(\Lambda,d)$ as the set of all irreducible components
of $\rep(\Lambda,d)$,
and $\Irr(\Lambda)$ as the disjoint union 
$\bigsqcup_{d \in (\Z_{\ge 0})^{Q_0}} \rep(\Lambda,d)$.
\end{Def}

For $d=(d_i)_{i \in Q_0} \in (\Z_{\ge 0})^{Q_0}$,
the representation scheme $\rep(\Lambda,d)$ admits the action of 
$\GL(d):=\prod_{i \in Q_0}\GL(d_i)$ given by 
\begin{align*}
g \cdot M:=(g_{t(\alpha)} \phi_\alpha g_{s(\alpha)}^{-1})_{\alpha \in Q_1}
\end{align*}
for any $M=(\phi_\alpha)_{\alpha \in Q_1} \in \rep(\Lambda,d)$
and $g=(g_i)_{i \in Q_0}$.

We set $\calO_M$ as the orbit of each $M \in \rep(\Lambda,d)$ 
with respect to $\GL(d)$.
Then we have 
\begin{align*}
\calO_M=\{X \in \rep(Q,d) \mid 
\text{$X \simeq M$ as $\Lambda$-modules} \}.
\end{align*}
Since $\GL(d)$ is a connected algebraic group,
each orbit $\calO_M$ is irreducible, 
so every irreducible component of $\rep(\Lambda,d)$ is a union of orbits.
In other words, if $\calZ \in \Irr(\Lambda)$ and $M \in \calZ$,
then the orbit $\calO_M$ and its closure $\overline{\calO_M}$
are contained in $\calZ$.
Moreover $\calO_M$ is locally closed in $\rep(\Lambda,d)$,
so $\calO_M$ is an open dense subset of $\overline{\calO_M}$.

The next notion is the point of our results and arguments.

\begin{Def}
Let $B \in \brick \Lambda$.
Then we say that $B$ is an \emph{open brick}
if  $\calO_B$ is an open dense subset 
of some $\calZ \in \Irr(\Lambda)$;
or equivalently, $\overline{\calO_B} \in \Irr(\Lambda)$ holds.
Otherwise, $B$ is called a \emph{non-open brick}.
\end{Def}

In particular, if $B \in \brick \Lambda$ satisfies 
$B \in \calZ_1 \cap \calZ_2$ 
for some distinct $\calZ_1,\calZ_2 \in \Irr(\Lambda)$,
then $B$ cannot be an open brick, because its orbit $\calO_B$
is contained in $\calZ_1 \cap \calZ_2$, which is a proper closed subset of
$\calZ_1$ and $\calZ_2$.

Typical examples of open bricks are labeling bricks 
between two functorially finite torsion classes in $\mod \Lambda$
\cite[Theorem 6.1]{MP3}. 
However, we do not use this fact in this paper.

\subsection{Proof of Theorem \ref{Thm_main}}

For any module $M \in \mod \Lambda$, we set two subcategories
\begin{align*}
{^\perp M}:=\{X \in \mod \Lambda \mid \Hom_\Lambda(X,M)=0\}, \quad
M^\perp:=\{X \in \mod \Lambda \mid \Hom_\Lambda(M,X)=0\}.
\end{align*}
From upper semicontinuity, we get the following properties.

\begin{Lem}\label{Lem_general}
Let $\calZ \in \Irr(\Lambda)$, $M \in \mod \Lambda$ and $t \in \Z_{\ge 0}$.
Then the following sets are open in $\calZ$:
\begin{align*}
\calZ \cap \brick \Lambda 
&:=\{X \in \calZ \mid X \in \brick \Lambda \}, \\
\calZ_{M,\le t}
&:=\{X \in \calZ \mid \dim_K \Hom_\Lambda(X,M) \le t\}, &
\calZ \cap {^\perp M}&:=
\{X \in \calZ \mid X \in {^\perp M}\}=\calZ_{M,\le 0}, \\
\calZ^{M,\le t}
&:=\{X \in \calZ \mid \dim_K \Hom_\Lambda(M,X) \le t\}, &
\calZ \cap M^\perp&:=\{X \in \calZ \mid X \in {M^\perp}\}=\calZ^{M,\le 0}.&
\end{align*} 
Thus $\calZ_{M,t}:=\calZ_{M,\le t} \setminus \calZ_{M,\le t-1}$ and
$\calZ^{M,t}:=\calZ^{M,\le t} \setminus \calZ^{M,\le t-1}$ 
are locally closed subsets of $\calZ$,
where we set $\calZ_{M,-1}=\calZ^{M,-1}:=\emptyset$.
\end{Lem}

Let $\calZ \in \Irr(\Lambda,d)$ and $M \in \rep(\Lambda,c)$
with $d,c \in (\Z_{\ge 0})^{Q_0}$.
For every $X \in \calZ$, we regard $\Hom_\Lambda(X,M)$ 
as a subset of 
$\Hom_K(K^d,K^c):=\prod_{i \in Q_0}\Hom_K(K^{d_i},K^{c_i})$.
For any subset $\calA \subset \calZ$, we set 
\begin{align*}
\calY(\calA,M)&:=
\{ (X,f) \in \calZ \times \Hom_K(K^d,K^c)
\mid X \in \calA, \ f \in \Hom_\Lambda(X,M) \}.
\end{align*}
Then we have the following.

\begin{Lem}\label{Lem_Bongartz}\cite[Lemma 2.1]{Bongartz}
Let $\calZ \in \Irr(\Lambda,d)$, $M \in \rep(\Lambda,c)$ 
with $d,c \in (\Z_{\ge 0})^{Q_0}$.
Then for each $t \in \Z_{\ge 0}$,
the set $\calY(\calZ_{M,t},M)$ is a locally closed subset of 
$\calZ \times \Hom_K(K^d,K^c)$, and the natural projection
\begin{align*}
\calY(\calZ_{M,t},M) \to \calZ_{M,t}
\end{align*}
is a vector bundle, and hence, an open map.
\end{Lem}

\begin{proof}
The proof of \cite[Lemma 2.1]{Bongartz} is valid also in this situation.
\end{proof}

Then we obtain the next property.

\begin{Prop}\label{Prop_GLFS}
Let $\calZ \in \Irr(\Lambda)$ and $B \in \calZ \cap \brick \Lambda$.
Then the following conditions are equivalent.
\begin{enumerate2}
\item
The set $\calZ \cap {^\perp B}$ is open dense in $\calZ$.
\item
The set $\calZ \cap {B^\perp}$ is open dense in $\calZ$.
\item
The orbit $\calO_B$ is not open dense in $\calZ$.
\end{enumerate2}
\end{Prop}

\begin{proof}
We use an argument similar to the proof of 
\cite[Theorem 1.5, (iv)$\Rightarrow$(i)]{GLFS}.
We only show (a)$\Leftrightarrow$(c), 
since (b)$\Leftrightarrow$(c) is dual.

Let (a) hold.
Then $\calO_B$ cannot be open dense, because otherwise,
$(\calZ \cap {^\perp B}) \cap \calO_B$ becomes open dense in $\calZ$ by (a),
which implies the contradiction $B \in {^\perp B}$.
Thus we get (c).

Conversely, assume that (a) does not hold.
Then $\calZ \cap {^\perp B}$ is empty by Lemma \ref{Lem_general}.
This and $B \in \brick \Lambda$ imply 
$\calZ_{B,\le 1}=\calZ_{B,1} \ne \emptyset$.
Thus $\calU:=\calZ_{B,1}$ is open dense in $\calZ$ 
by Lemma \ref{Lem_general} again.
Then Lemma \ref{Lem_Bongartz} implies that the natural projection
\begin{align*}
p \colon \calY(\calU,B) \to \calU
\end{align*}
is an open map.

Define $\calV:=\{(X,f) \in \calY(\calU,B) \mid 
\text{$f$ is an isomorphism}\}$.
Then $\calV$ is open in $\calY(\calU,B)$, and 
it is nonempty by $(B,1_B) \in \calV$, where $1_B$ is the identity map.
Thus $p(\calV)$ is open and nonempty in $\calU$,
so $p(\calV)$ is open dense in $\calZ$.
On the other hand, $p(\calV)=\calO_B$ holds by definition.
Thus $\calO_B$ is open dense in $\calZ$, so (c) does not hold as desired.
\end{proof}

Now we can prove the following property.

\begin{Prop}\label{Prop_perp}
Let $\Lambda$ be a finite dimensional algebra 
over an algebracally closed field $K$, and $M,N \in \mod \Lambda$.
Assume that $B \in \brick \Lambda$ and $\calZ \subset \Irr(\Lambda)$
satisfy $B \in {^\perp M} \cap N^\perp$ and 
$\overline{\calO_B} \subsetneq \calZ$.
Then there exists some brick $B'$ such that
\begin{align*}
B' \in \calZ \cap \brick \Lambda 
\cap {^\perp (M \oplus B)} \cap {(N \oplus B)^\perp}. 
\end{align*}
\end{Prop}

\begin{proof}
Since $B$ is a brick such that $\overline{\calO_M} \subsetneq \calZ$, 
both $\calZ \cap {^\perp B}$ and $\calZ \cap B^\perp$ are open dense
by Proposition \ref{Prop_GLFS}.
Moreover $B \in {^\perp M} \cap N^\perp$ and $B \in \brick \Lambda$
imply that $\calZ \cap {^\perp M}$,
$\calZ \cap N^\perp$ and $\calZ \cap \brick \Lambda$ are nonempty.
These sets are open dense in $\calZ$ by Lemma \ref{Lem_general}.
Therefore we get
\begin{align*}
\calZ \cap \brick \Lambda \cap {^\perp (M \oplus B)} 
\cap {(N \oplus B)^\perp} \ne \emptyset.
\end{align*}
Taking $B'$ from this set, we have the assertion.
\end{proof}

Then Theorem \ref{Thm_main} follows.

\begin{proof}[Proof of Theorem \ref{Thm_main}]
This is obtained by setting 
$M$ and $N$ as $\bigoplus_{X \in \calS \setminus \{B\}} X$
in Proposition \ref{Prop_perp},
because $\calS$ is finite.
\end{proof}

As in Section \ref{Sec_result}, Theorem \ref{Thm_main} implies 
Theorem \ref{Thm_open} and Corollary \ref{Cor_infin}. 
\qed

\medskip

We end this section with an observation on brick components.
As in \cite{GLFS}, 
an irreducible component $\calZ \in \Irr(\Lambda)$ is called 
a \emph{brick component} if $\calZ \cap \brick \Lambda \ne \emptyset$.
In this case, $\calZ \cap \brick \Lambda$ is open dense in $\calZ$ 
by Lemma \ref{Lem_general}.

\begin{Prop}\cite[Proposition 3.4]{MP2}\label{Prop_MP}
Let $\calZ \in \Irr(\Lambda)$ be a brick component.
Then exactly one of the following holds.
\begin{enumerate2}
\item
There exists an open brick $B \in \calZ \cap \brick \Lambda$ such that 
$\calZ \cap \brick \Lambda=\calO_B$;
in other words, $B$ is the unique brick in $\calZ$ up to isomorphisms.
\item
The set $\calZ \cap \brick \Lambda$ is a union of
infinitely many non-open bricks.
\end{enumerate2}
\end{Prop}

Thus open bricks are nothing but the bricks appearing in (a).

In \cite{MP2}, Proposition \ref{Prop_MP} is proved 
by using \cite[Remark 3.3]{MP2}; namely, 
every $B \in \calZ \cap \brick \Lambda$ satisfies
$\dim_K \calO_B=\dim_K \GL(d)-1$.
We give another proof of this proposition 
with Proposition \ref{Prop_GLFS}.

\begin{proof}
It suffices to show the following statement:
\begin{quote}
if there exists an open brick $B$ in $\calZ$,
then $\calZ \cap \brick \Lambda \subset \calO_B$.
\end{quote}
Let $B$ be an open brick in $\calZ$.
Assume that there exists $B' \in \calZ \cap \brick \Lambda$
such that $B' \notin \calO_B$.
Then since $\calO_B$ is open dense in $\calZ$,
the orbit $\calO_{B'} \subset \calZ$ cannot be open dense in $\calZ$,
so Proposition \ref{Prop_GLFS} implies that
$\calZ \cap {^\perp (B')}$ is open dense in $\calZ$.
This set is a union of orbits, so must contain the open dense orbit $\calO_B$.
Thus we get $B \in {^\perp (B')}$, which implies $B' \in B^\perp$,
so Lemma \ref{Lem_general} gives $\calZ \cap B^\perp$ 
is open dense in $\calZ$.
By Proposition \ref{Prop_GLFS} again,
$\calO_B$ is not open dense in $\calZ$, a contradiction.
\end{proof}

\section{Proposition \ref{Prop_perp} for path algebras}\label{Sec_old}

In this section, we record our first proof of 
a weaker version of Proposition \ref{Prop_perp} for path algebras,
since it is an interesting one 
relying on the well-established theory of semi-invariants \cite{King, SV, DW}
and presentation spaces \cite{DF}.
Theorem \ref{Thm_open} for path algebras follows also from this.

We assume that $\Lambda=KQ$ for a finite acyclic quiver $Q$ in the rest.
We say that a brick $B \in \brick \Lambda$ is \emph{exceptional}
if $\Ext_\Lambda^1(B,B)=0$ as in \cite{CB2}.
Since $\Lambda$ is a path algebra,
a brick $B \in \brick \Lambda$ is an open brick if and only if 
$B$ is exceptional by Voigt's Lemma \cite{Gabriel,Voigt}.

We aim to show the following weaker version of Proposition \ref{Prop_perp}.

\begin{Prop}\label{Prop_perp_weak}
Let $\Lambda=KQ$ with $Q$ a finite acyclic quiver,
and $M,N \in \mod \Lambda$.
Assume that $B \in {^\perp M} \cap N^\perp$
and that $B \in \brick \Lambda$ is not exceptional.
Then there exists some $B' \in \brick \Lambda$ such that
\begin{align*}
B' \in {^\perp (M \oplus B)} \cap (N \oplus B)^\perp, \quad
\dimv B' \in \Z_{\ge 1} \cdot \dimv B.
\end{align*}
\end{Prop}

\subsection{Preliminaries}

In this subsection, 
we recall some additional properties on quiver representations.

Since $\Lambda=KQ$, the representation scheme $\rep(\Lambda,d)$
for $d \in (\Z_{\ge 0})^{Q_0}$ is just $\rep(Q,d)$.
It is the direct product 
$\prod_{\alpha \in Q_1}\Hom_K(K^{s(\alpha),t(\alpha)})$,
so in particular, irreducible.
Thus $\Irr(\Lambda)$ is identified with $(\Z_{\ge 0})^{Q_0}$.

We define Schur roots as follows.

\begin{Def}\label{Def_Schur}
Let $d \in (\Z_{\ge 0})^{Q_0}$.
Then we call $d$ a \emph{Schur root} 
if $\rep(Q,d) \cap \brick \Lambda$ is nonempty.
\end{Def}

By Lemma \ref{Lem_general}, if $d$ is a Schur root, 
then $\rep(Q,d) \cap \brick \Lambda$ is open dense in $\rep(Q,d)$.

Schur roots depend on the orientation of the quiver $Q$, 
but do not depend on the field $K$.
This follows from \cite[Proposition 6.1]{Schofield} 
together with the correction \cite{CB} of \cite[Proposition 3.3]{Schofield}.

We recall the following basic property of path algebras.

\begin{Lem}\label{Lem_Schur_indec}\cite[Proposition 1 (a)]{Kac2}
Let $d \in (\Z_{\ge 0})^{Q_0}$.
Then the following conditions are equivalent.
\begin{enumerate2}
\item
The element $d$ is a Schur root.
\item
The set $\{X \in \rep(Q,d) \mid \text{$X$ is indecomposable}\}$ contains
an open dense subset of $\rep(Q,d)$.
\end{enumerate2}
\end{Lem}

By using Schur roots, 
we can consider canonical decompositions of dimension vectors.

\begin{Def-Prop}\label{Def-Prop_canon}
\cite[Subsection 2.8 a)]{Kac} 
Let $d\in (\Z_{\ge 0})^{Q_0}$.
Then there exist Schur roots $d_1,\ldots,d_m \in (\Z_{\ge 0})^{Q_0}$
such that the following set contains an open dense subset of $\rep(Q,d)$:
\begin{align*}
\{X \in \rep(Q,d) \mid 
\text{$X \simeq \bigoplus_{i=1}^m X_i$ with $\dimv X_i=d_i$
and $X_i$ is indecomposable for each $i$}\}.
\end{align*}
Moreover such $d_1,d_2\ldots,d_m$ are unique up to reordering.
In this case, we write $d=\bigoplus_{i=1}^m d_i$,
and call it the \emph{canonical decomposition} of $d$.
\end{Def-Prop}

The following trichotomy of Schur roots will play an important role.
Here we used the characterization \cite[Proposition 3]{Kac2}
of canonical decompositions.

\begin{Def}
Let $d \in (\Z_{\ge 0})^{Q_0}$ be a Schur root.
\begin{enumerate}
\item
We say that $d$ is \emph{real} 
if there exists some exceptional module $X \in \rep(Q,d)$.
\item
We say that $d$ is \emph{tame}
if $d$ is not real but the canonical decomposition of $2d$ is $d \oplus d$.
\item
We say that $d$ is \emph{wild} if $d$ is neither real nor tame.
\end{enumerate}
\end{Def}

Let $d$ be a Schur root.
If $d$ is real, 
then there uniquely exists a brick $B \in \rep(Q,d)$ up to isomorphisms
by \cite[Theorem 1]{Kac}, and $B$ is exceptional.
Thus if there exists a brick $B \in \rep(Q,d)$ which is not exceptional, 
then $d$ is tame or wild.

Canonical decompositions of multiples of Schur roots are given as follows.

\begin{Prop}\cite[Theorem 3.8]{Schofield}\label{Prop_Schofield}
Let $d \in (\Z_{\ge 0})^{Q_0}$ be a Schur root and $l \in \Z_{\ge 1}$.
Then the canonical decomposition of $ld$ is
$d^{\oplus l}$ if $d$ is real or tame, and $ld$ if $d$ is wild.
\end{Prop}

\begin{proof}
The original proof of \cite[Theorem 3.3]{Schofield}
works in the case that the characteristic of $K$ is zero,
but it still holds for any algebraically closed field
thanks to Crawley-Boevey's correction \cite{CB}.
Then the proofs of \cite[Theorems 3.5, 3.7, 3.8]{Schofield} proceed.
\end{proof}

\subsection{Proof of Proposition \ref{Prop_perp_weak}}

We will use the Grothendieck group $K_0(\mod \Lambda)$ here.

For each vertex $i \in Q_0$,
we write $S_i$ for the simple $\Lambda$-module corresponding to $i$,
and $[S_i]$ for its isoclass.
Then the Grothendieck group $K_0(\mod \Lambda)$
is the free abelian group $\bigoplus_{i \in Q_0} \Z[S_i]$,
so we get $K_0(\mod \Lambda) \simeq \Z^{Q_0}$ as abelian groups.
Under this isomorphism, for each module $M \in \mod \Lambda$,
the element $[M] \in K_0(\mod \Lambda)$ corresponds 
to its dimension vector $\dimv M \in (\Z_{\ge 0})^{Q_0}$.
Thus in what follows, the dimension vector $\dimv M$ is regarded
as an element of $K_0(\mod \Lambda)_{\ge 0}:=
\sum_{i \in Q_0} \Z_{\ge 0}[S_i]$.

We also use the (split) Grothendieck group $K_0(\proj \Lambda)$
of the category $\proj \Lambda$ of finitely generated projective modules.
For each $i \in Q_0$, we write $P_i$ 
for the projective cover of the simple module $S_i$.
Then $K_0(\proj \Lambda)$ 
has the isoclasses of indecomposable projective modules
$[P_i]$ for $i \in Q_0$ as a $\Z$-basis,
so again we have 
$K_0(\proj \Lambda)=\bigoplus_{i \in Q_0} \Z[P_i] \simeq \Z^{Q_0}$.

From now on, a bracket $[\quad]$ will be used only to denote
an element of $K_0(\proj \Lambda)$.
Elements of $K_0(\mod \Lambda)$ will be expressed 
like $\dimv M$ or $\dimv M-\dimv N$ by using dimension vectors of modules.

Via the \emph{Euler bilinear form}
$\langle !,? \rangle \colon 
K_0(\proj \Lambda) \times K_0(\mod \Lambda) \to \Z$
given by $\langle [P_i],\dimv S_j \rangle=\delta_{i,j}$,
each element $\theta \in K_0(\proj \Lambda)$ induces a $\Z$-linear map 
$\theta:=\langle \theta,? \rangle \colon K_0(\mod \Lambda) \to \Z$.
In this way, 
we have $K_0(\proj \Lambda) \simeq \Hom_\Z(K_0(\mod \Lambda),\Z)$.
For a module $M \in \mod \Lambda$,
we set $\theta(M):=\theta(\dimv M) \in \Z$ in short.

To each $\theta \in K_0(\proj \Lambda)$,
Derksen-Fei \cite{DF} associated the \emph{presentation space} $\Hom(\theta)$.

\begin{Def}\cite[Definition 1.1]{DF}
Let $\theta \in K_0(\proj \Lambda)$.
\begin{enumerate}
\item
We define $P^0_\theta,P^1_\theta \in \proj \Lambda$ up to isomorphisms
so that $\theta=[P^0_\theta]-[P^1_\theta]$ and 
$\add P^0_\theta \cap \add P^1_\theta=\{0\}$.
\item
We define an irreducible algebraic variety 
$\Hom(\theta):=\Hom_\Lambda(P_1,P_0)$,
and call it the \emph{presentation space} of $\theta$.
\item
For each $f \in \Hom(\theta)$,
a 2-term complex $P_f$ is defined 
by $P_f:=(P_1^\theta \xrightarrow{f} P_0^\theta)$
in the homotopy category $\sfK^\rmb(\proj \Lambda)$ 
of bounded complexes over $\proj \Lambda$.
\end{enumerate}
\end{Def}

Then the value $\theta(M)$ can be calculated as follows,
where $\sfD^\rmb(\mod \Lambda)$ is 
the bounded derived category over $\mod \Lambda$.
The following is well-known and elementary.

\begin{Lem}\label{Lem_theta(M)_nu}
Let $\theta \in K_0(\proj \Lambda)$ and $f \in \Hom(\theta)$. 
Take the maximum $P \in \proj \Lambda$ such that 
$P[1]$ is a direct summand of $P_f$.
Then for any $M \in \mod \Lambda$, we have
\begin{align*}
\theta(M)&=\dim_K \Hom_{\sfD^\rmb(\mod \Lambda)}(P_f,M)
-\dim_K \Hom_{\sfD^\rmb(\mod \Lambda)}(P_f,M[1])\\
&=\dim_K \Hom_\Lambda(\Coker f,M)-
\dim_K \Ext_\Lambda^1(\Coker f,M)-\dim_K \Hom_\Lambda(P,M).
\end{align*}
\end{Lem}

We define a $\Z$-linear map 
\begin{align*}
\iota \colon K_0(\proj \Lambda) \to K_0(\mod \Lambda)
\end{align*}
so that $[P] \in K_0(\proj \Lambda)$ is sent to 
$\dimv P \in K_0(\mod \Lambda)$
for every indecomposable projective module $P \in \proj \Lambda$.

For any $\theta \in K_0(\proj \Lambda)$, the set
$\{ f \in \Hom(\theta) \mid \text{$f$ is injective} \}$ is an open subset
of $\Hom(\theta)$.
Thus we have the following property.

\begin{Lem}\label{Lem_Hom_inj}
Let $d \in (\Z_{\ge 0})^{Q_0} \simeq K_0(\mod \Lambda)_{\ge 0}$, 
and set $\theta:=\iota^{-1}(d)$.
\begin{enumerate}
\item
Let $N \in \rep(Q,d)$. 
Taking its minimal projective presentation $0 \to P_1 \to P_0 \to N \to 0$,
we get $\iota^{-1}(d)=[P_0]-[P_1]$.
\item
For any $l \in \Z_{\ge 1}$,
the set $\{ f \in \Hom(l \theta) \mid \text{$f$ is injective} \}$
is an open dense subset of $\Hom(l \theta)$.
\item
Let $N \in \rep(Q,d)$ and $M \in \mod \Lambda$.
Then we have 
\begin{align*}
\theta(M)=\dim_K \Hom_\Lambda(N,M)-\dim_K \Ext_\Lambda^1(N,M).
\end{align*}
\end{enumerate}
\end{Lem}

\begin{proof}
(1)
By the short exact sequence $0 \to P_1 \to P_0 \to N \to 0$,
we have $d=\dimv N=\dimv P_0-\dimv P_1=\iota([P_0])-\iota([P_1])$ 
in $K_0(\mod \Lambda)$.
Since $\iota$ is an isomorphism, we have the assertion.

(2)
By upper semicontinuity,
we know that $\{f \in \Hom(\theta) \mid \text{$f$ is injective}\}$ is open
in $\Hom(l\theta)$, so will show that it is nonempty.
Take $N \in \rep(Q,d)$
and its minimal presentation $0 \to P_1 \xrightarrow{f} P_0 \to N \to 0$.
Then $f^{\oplus l}$ is injective.
Thus $\{f \in \Hom(\theta) \mid \text{$f$ is injective}\}$ is nonempty;
hence it is open dense in $\Hom(l\theta)$.

(3)
Take the minimal presentation $0 \to P_1 \xrightarrow{f} P_0 \to N \to 0$.
Then $f$ is injective, so Lemma \ref{Lem_theta(M)_nu} implies the assertion.
\end{proof}

We use the following subcategory associated to each $\theta$.

\begin{Def}\cite[Subsection 3.1]{BKT}
Let $\theta \in K_0(\proj \Lambda)$.
Then we set
\begin{align*}
\ovcalF_\theta&:=\{M \in \mod \Lambda \mid 
\text{any submodule $L$ of $M$ satisfies $\theta(L) \le 0$}\}.
\end{align*}
\end{Def}

The subcategory $\ovcalF_\theta$ is a torsion-free class in $\mod \Lambda$,
that is, $\ovcalF_\theta$ is closed under taking extensions and submodules.
However, we do not use this property directly.
We have the following observation.

\begin{Lem}\label{Lem_ovcalF_theta}
Assume that $B \in \brick \Lambda$ is not exceptional.
Set $\theta:=\iota^{-1}(\dimv B) \in K_0(\proj \Lambda)$.
Then we have $B \in \ovcalF_\theta$.
\end{Lem}

\begin{proof}
Since $B$ is a brick, we get $\dim_K \Hom_\Lambda(B,B)=1$.
Moreover, since $B$ is not exceptional, 
we obtain $\dim_K \Ext_\Lambda^1(B,B) \ge 1$.
Then Lemma \ref{Lem_Hom_inj} (3) gives $\theta(B) \le 0$.
For any proper submodule $L$ of $B$, we have $\Hom_\Lambda(B,L)=0$,
since $B$ is a brick.
Then $\theta(L) \le 0$ follows again from Lemma \ref{Lem_Hom_inj} (3).
Thus we get $B \in \ovcalF_\theta$.
\end{proof}

The following property from semi-invariants is important.

\begin{Prop}\label{Prop_Fei}
\cite[Theorem 3.6]{Fei2}
\cite[Theorem 4.3]{AsI1}
\cite[Proposition 4.10]{Garcia}
Let $\theta \in K_0(\proj \Lambda)$ and $M \in \ovcalF_\theta$.
Then there exists some $l \in \Z_{\ge 1}$ such that
$\{ f \in \Hom_\Lambda(l \theta) \mid \Hom_\Lambda(\Coker f,M)=0 \}$
is an open dense subset of $\Hom_\Lambda(l \theta)$.
\end{Prop}

Lemma \ref{Lem_ovcalF_theta} and Proposition \ref{Prop_Fei} give 
the next properties.

\begin{Prop}\label{Prop_X_B_weak}
Let $\Lambda=KQ$.
Assume that $B \in \brick \Lambda$ is not exceptional.
\begin{enumerate}
\item
There exists $X \in \mod \Lambda$ satisfying 
$X \in {^\perp B}$ and $\dimv X \in \Z_{\ge 1} \cdot \dimv B$.
\item
There exists $X \in \mod \Lambda$ satisfying 
$X \in B^\perp$ and $\dimv X \in \Z_{\ge 1} \cdot \dimv B$.
\end{enumerate}
\end{Prop}

\begin{proof}
(1)
We set $d:=\dimv B$ and $\theta:=\iota^{-1}(d) \in K_0(\proj \Lambda)$.
Since $B$ is a brick which is not exceptional, 
Lemma \ref{Lem_ovcalF_theta} implies $B \in \ovcalF_\theta$.

By Proposition \ref{Prop_Fei},
we can take $l \in \Z_{\ge 1}$ such that
$\{ f \in \Hom_\Lambda(l \theta) \mid \Hom_\Lambda(\Coker f,B)=0 \}$
is an open dense subset of $\Hom_\Lambda(l \theta)$.
Moreover Lemma \ref{Lem_Hom_inj} (2) implies that
$\{ f \in \Hom_\Lambda(l \theta) \mid \text{$f$ is injective} \}$
is an open dense subset of $\Hom_\Lambda(l \theta)$.
Therefore we can take $f \in \Hom_\Lambda(l \theta)$ which is injective
and satisfies $\Hom_\Lambda(\Coker f,B)=0$.

Set $X:=\Coker f$.
Then $f$ is a minimal projective presentation of $X$,
so we have $\iota^{-1}(\dimv X)=l \theta$ in $K_0(\proj \Lambda)$
by Lemma \ref{Lem_Hom_inj} (1).
Thus we get $\dimv X=\iota(l \theta)=ld=l \cdot \dimv B$ 
in $K_0(\mod \Lambda)$.
Moreover $\Hom_\Lambda(X,B)=\Hom_\Lambda(\Coker f,B)=0$ are 
nothing but $X \in {^\perp B}$.

(2) is dually shown.
\end{proof}

Now we can prove Proposition \ref{Prop_perp_weak}.

\begin{proof}[Proof of Proposition \ref{Prop_perp_weak}]
Set $d:=\dimv B$.
Since $B$ is not exceptional,
Proposition \ref{Prop_X_B_weak} allows us 
to take $X_1,X_2 \in \mod \Lambda$ and $l_1,l_2 \in \Z_{\ge 1}$
with $X_1 \in {^\perp B}$, $X_2 \in {B^\perp}$ and 
$\dimv X_i=l_i d$ ($i \in \{1,2\}$).

We first show that 
\begin{align*}
\rep(Q,l_1l_2d) \cap {^\perp (M \oplus B)} \cap {(N \oplus B)^\perp}
\end{align*}
is open dense in $\rep(Q,l_1l_2d)$.

Since $B \in {^\perp M} \cap {N^\perp}$,
the module $B^{\oplus l_1l_2}$ belongs 
to both $\rep(Q,l_1l_2d) \cap {^\perp M}$
and $\rep(Q,l_1l_2d) \cap {N^\perp}$.
Then by Lemma \ref{Lem_general}, 
these sets are open dense in $\rep(Q,l_1l_2d)$.

Moreover $X_1 \in {^\perp B}$, $X_2 \in {B^\perp}$ and 
$\dimv X_i=l_i d$ ($i \in \{1,2\}$) imply
$X_1^{\oplus l_2} \in {^\perp B}$, $X_2^{\oplus l_1} \in {B^\perp}$ and
$\dimv X_1^{\oplus l_2}=\dimv X_2^{\oplus l_1}=l_1l_2d$.
Thus the sets $\rep(Q,l_1l_2d) \cap {^\perp B}$ 
and $\rep(Q,l_1l_2d) \cap {B^\perp}$
are open dense in $\rep(Q,l_1l_2d)$ by Lemma \ref{Lem_general} again.

Thus 
$\rep(Q,l_1l_2d) \cap {^\perp (M \oplus B)} \cap {(N \oplus B)^\perp}$
is open dense.

Since $B \in \brick \Lambda$ is not exceptional, 
$d$ is a Schur root which is tame or wild.
Set $l:=1$ if $d$ is tame, and $l:=l_1l_2$ if $d$ is wild.
It suffices to show that
\begin{align*}
\calX:=\rep(Q,ld) \cap \brick \Lambda \cap {^\perp (M \oplus B)} 
\cap {(N \oplus B)^\perp}
\end{align*}
is nonempty.

If $d$ is tame, then Proposition \ref{Prop_Schofield} 
and Definition-Proposition \ref{Def-Prop_canon} imply that
there exists an open dense subset $\calU \subset \rep(Q,l_1l_2d)$ such that 
\begin{align*}
\calU \subset \{X \in \rep(Q,l_1l_2d) \mid 
\text{$X \simeq \bigoplus_{i=1}^{l_1l_2} X_i$ with $\dimv X_i=d$
and $X_i$ is indecomposable for each $i$}\}.
\end{align*}
Thus $\rep(Q,l_1l_2d) \cap {^\perp (M \oplus B)} \cap {(N \oplus B)^\perp} 
\cap \calU$ is open dense in $\rep(Q,l_1l_2d)$.
Take $X$ from this set, and take an indecomposable direct summand $Y$ of $X$.
Then $Y \in {^\perp (M \oplus B)} \cap {(N \oplus B)^\perp}$ 
follows from $X \in {^\perp (M \oplus B)} \cap {(N \oplus B)^\perp}$,
and $\dimv Y=d$ comes from $X \in \calU$.
Thus $Y$ belongs to 
$\rep(Q,d) \cap {^\perp (M \oplus B)} \cap {(N \oplus B)^\perp}$,
so this set is open dense in $\rep(Q,d)$ by Lemma \ref{Lem_general}.
Moreover since $d$ is a Schur root, $\rep(Q,d) \cap \brick \Lambda$
is open dense in $\rep(Q,d)$ by Lemma \ref{Lem_general}.
Thus the set $\calX \subset \rep(Q,d)$ is open dense, and thus, nonempty.

If $d$ is wild, then $l_1l_2d$ is a Schur root 
by Proposition \ref{Prop_Schofield},
so $\rep(Q,l_1l_2d) \cap \brick \Lambda$ is open dense in $\rep(Q,l_1l_2d)$.
Since $\rep(Q,l_1l_2d) \cap {^\perp (M \oplus B)} \cap {(N \oplus B)^\perp}$
is open dense in $\rep(Q,l_1l_2d)$, so is $\calX$ as desired.
\end{proof}

Proposition \ref{Prop_perp_weak} implies 
the following weaker version of Theorem \ref{Thm_main} for path algebras.

\begin{Thm}\label{Thm_main_weak}
Let $\Lambda=KQ$, 
and $\calS \in \sbrick \Lambda$ be a finite semibrick.
Assume that a brick $B \in \calS$ is not exceptional.
Then there exists some brick $B'$ satisfying
\begin{align*}
\calS \sqcup \{B'\} \in \sbrick \Lambda, \quad
\dimv B' \in \Z_{\ge 1} \cdot \dimv B.
\end{align*}
\end{Thm}

Then Theorem \ref{Thm_open} for path algebras follows 
from Theorem \ref{Thm_main_weak}.
\qed

\section*{Acknowlegments}
The author deeply thanks Yuki Hirano for asking him 
whether Theorem \ref{Thm_open} for path algebras 
related to \cite[Theorem 3.9]{HKO}
holds at MSJ Autumn Meeting 2024 in Osaka and giving him useful comments.
He appreciates the helpful comments of an anonymous referee.
He is grateful also to William Crawley-Boevey, Osamu Iyama,
Kaveh Mousavand and Charles Paquette for improving drafts.
The author was supported by JSPS KAKENHI Grant Number JP23K12957.

\end{document}